\RequirePackage{fix-cm}
\documentclass[smallcondensed]{svjour3}     
\smartqed  
\usepackage{graphicx}
\usepackage{stmaryrd}
\usepackage{amsfonts}
\usepackage{amsmath,amssymb}
\usepackage{mathrsfs}
\usepackage{amstext}
\usepackage{subeqnarray}
\usepackage{txfonts,graphics,epsfig,url,cases,fancyvrb}
  \usepackage{color}
\newcommand{\px}[1][x]{\partial_{#1}}
\newcommand{\dx}[1][x]{\,{\rm d}#1}

\begin{document}

\title{Second-order   Stable Finite  Difference Schemes for the Time-fractional Diffusion-wave Equation
\thanks{This research is supported by National Natural Science Foundation of China (No. 11171256).}
}


\titlerunning{Second-order schemes for fractional wave equation}        

\author{Fanhai Zeng}

%
\institute{Department of Mathematics, Tongji University, Shanghai 200092, P. R. China and
Division of Applied Mathematics, Brown University, Providence RI, 02912
\\
              \email{fanhaiz@shu.edu.cn}           
}

\date{Received: date / Accepted: date}

\maketitle

\begin{abstract}
We propose two   stable and one conditionally stable finite difference
schemes of   second-order in both time and space for the time-fractional diffusion-wave equation.
In the first scheme, we apply the fractional trapezoidal rule in time and the central difference in space.
We  use the generalized Newton-Gregory formula  in time
for the  second  scheme   and its modification for the third scheme. While the second scheme is conditionally stable,
the first and the third  schemes are  stable.
We  apply  the methodology to the considered equation with also  linear advection-reaction terms
and also obtain    second-order schemes both in time and space.
Numerical examples with  comparisons among  the  proposed   schemes and the existing  ones verify the theoretical analysis and
show that the present schemes exhibit better performances than the known ones.
\end{abstract}

\section{Introduction}\label{intro}
In this work, we consider second-order finite difference schemes in both time and space for  the following
time-fractional diffusion-wave equation, see e.g. \cite{LuchkoMainardi2013,Mainardi95,MetzlerNonnenmacher2002,SchneiderWyss89,Vazquez2011},
\begin{equation}\label{eq:subeq}
\left\{\begin{aligned}
&{}_{C}D_{0,t}^{\beta}U(x,t)=\mu\,\px^2U(x,t)+f(x,t),
{\quad}(x,t){\,\in\,}I{\times}(0,T],I=(a,b),T>0,\\
&U(x,0)=\phi_0(x),{\quad}\px[t]U(x,0)=\psi_0(x),{\quad}x{\,\in\,}I,\\
&U(a,t)=U_a(t),{\quad}U(b,t)=U_b(t),{\quad}t{\,\in\,}(0,T],
\end{aligned}\right.
\end{equation}
where $1<\beta<2,\,\mu>0$,
and ${}_{C}D_{0,t}^{\beta}$  is
the $\beta$th-order Caputo derivative operator defined by
\begin{equation}\label{eq:cpt}
{}_{C}D_{0,t}^{\beta}U(x,t)=D^{-(2-\beta)}_{0,t}\left[{\partial_t^2}U(x,t)\right]
=\frac{1}{\Gamma(2-\beta)}\int_{0}^t(t-s)^{1-\beta}{\px[s]^2U(x,s)}\dx[s],
\end{equation}
in which $D^{-\gamma}_{0,t}$ is the fractional integral operator defined by, see for example \cite{Podlubny1999},
\begin{equation}\label{eq:fint}
D^{-\gamma}_{0,t}U(x,t)={}_{RL}D^{-\gamma}_{0,t}U(x,t)
=\frac{1}{\Gamma(\gamma)}\int_{0}^t(t-s)^{\gamma-1}{U(x,s)}\dx[s],{\quad}\gamma>0.
\end{equation}

Temporal finite difference schemes for the time-fractional diffusion-wave equation \eqref{eq:subeq} and its equivalent form
are  mostly  of first-order, ($3-\beta$)-order and second-order convergence in time.
First-order  schemes can be  based on either  $L2$ method and its generalization, see e.g.  \cite{MurilloYuste2011,MurilloYuste2013}  or
the first- and second-order  fractional backward difference methods, see e.g. \cite{HuangTang2013,SweilamKhaderAdel2012}.
The $(3-\beta)$th-order time discretization techniques are based on the L1 method   \cite{Sun06}, see also
\cite{DuCaoSun2010,LiXuLuo2013,RenSun2013,ZhangSun2012}.
Second-order schemes are  either  generalized Crank--Nicolson schemes \cite{McLeanMustapha2007}
or based on  fractional backward difference methods,
see \cite{CuestaLubich2006,DingLi2013,JinLazarovZhou2014}.
In \cite{YangHuang2014}, the $\beta$th-order method was derived based on the Crank--Nicolson scheme and the
second-order  fractional backward difference method.
There have existed other related  works on the time-fractional diffusion equations, see e.g.
\cite{BhrawyDoha2014,ChenLiu-etal2012,Hanygad2002,JafariAminataei2011,JafariMomani2007,FLiu2013,MaoXiao2014,MustaphaMcLean2013}.

In this paper, we adopt different time discretization approaches to the time-fractional diffusion-wave equation of the
form \eqref{eq:subeq}, which yields three schemes with second-order accuracy both in time and space.
The  key of our discretization is that  we use three different second-order generating functions for the time discretization of
\eqref{eq:subeq} which are different from those in all the aforementioned works.
Our first scheme for time discretization is based on the second-order fractional trapezoidal rule as that used in \cite{ZengLi2013a}.
The second and third schemes   are based on  the second-order generalized Newton-Gregory formula in time
and its modification.
With the second-order central difference method in space discretization, we can prove
that the first and the third schemes
are    stable and that  the second one is conditionally stable  through the Fourier analysis, and
all the schemes are convergent  of  order two both in time and space.

%
%
%
 %
 %

One important feature of our schemes is that
these schemes
respectively reduce to classical difference schemes
when $\beta\to2$ 
while the second-order schemes in
\cite{CuestaLubich2006,DingLi2013,JinLazarovZhou2014,McLeanMustapha2007} do  not.
{Specifically, when $\beta\to 2$, our discretization in  time for \eqref{eq:subeq} (with $f(x,t)=0$)
is reduced to $\frac{1}{\tau^2}\left({u_j^{n+1}-2u_j^n+u_j^{n-1}}\right)
=\frac{\mu}{4}\left(\delta_x^2u_j^{n+1}+2\delta_x^2u^{n}_j+\delta_x^2u^{n-1}_j\right)$
or
$\frac{1}{\tau^2}\left({u_j^{n+1}-2u_j^n+u_j^{n-1}}\right)
=\frac{\mu}{2}\left(\delta_x^2u_j^{n+1}+ \delta_x^2u^{n-1}_j\right)$.
In other words,  our schemes are extensions of classical  central difference in time.}
While in \cite{DingLi2013,JinLazarovZhou2014,YangHuang2014},
the second-order fractional backward difference method is
used to discretize the time-fractional derivative, which can not be reduced to any form of our schemes.
For example,  when $\beta\to2$ and $f(x,t)=0$,
the method in  \cite{DingLi2013,JinLazarovZhou2014}    is reduced to
{\[
 \frac{1}{\tau^2}\left({\frac{9}{4}u_j^{n+1}-6u_j^n
 +\frac{11}{2}u_j^{n-1}-2u_{j}^{n-2}+\frac{1}{4}u_j^{n-3}}\right)
 ={\mu}\delta_x^2u_j^{n+1},\quad n\geq 3.\]}
%
The second-order schemes in \cite{CuestaLubich2006,McLeanMustapha2007}  do not  reduce  to the central difference scheme   in time  when $\beta=2$.
Though  the discretization of time derivative in \cite{SweilamKhaderAdel2012} can lead to  the central difference scheme  in time for $1<\beta\leq 2$, the method in \cite{SweilamKhaderAdel2012} has only first-order accuracy in time for $1<\beta<2$.

{Spatial discretization for  \eqref{eq:subeq}}  can be finite difference methods,
see e.g.  {\cite{DuCaoSun2010,LiXuLuo2013,RenSun2013,Sun06,ZhangSun2012}}
and finite element methods,
see e.g. \cite{JinLazarovZhou2014,McLeanMustapha2007}.  Here, we consider finite difference methods while the finite
element methods can be also applied.

The remainder of this paper is outlined as follows. In Section 2, we present a fully discrete finite difference scheme
for \eqref{eq:subeq} and establish the analysis of the stability, consistency, and convergence. In Section 3,
we propose  two more fully schemes  for \eqref{eq:subeq}, one is conditionally stable and the other is  stable.
We present numerical schemes for the time-fractional diffusion-wave equation
with linear advection-reaction term  in  Section 4.
Numerical experiments are provided in Section 5 before the conclusion  in the last section.

\section{The finite difference scheme based on the fractional trapezoidal rule} \label{sec:1}
In this section, we first present the time discretization for \eqref{eq:subeq}  based on the fractional trapezoidal rule.
With  space discretization by the central
finite difference, we prove the stability,  consistency, and convergence
of the  fully  discrete scheme.

\subsection{The finite difference scheme in time}
Let $\tau$ be the time step size and $n_T$ be a positive integer
with $\tau=T/n_T$ and $t_n=n\tau$ for $n=0,1,...,n_T$. For the
function $y(t)\in C([0,T])$, denote by $y^n=y(t_n)$. Denote by $h$ as the space step size with
$h=(b-a)/N$, where $N$ is a positive integer.  The space grid point $x_j$
is defined as $x_j=a+jh,j=0,1,...,N$. For the function {$u(x,t){\,\in\,}C(C(\bar{I});[0,T])$},
we also denote by
$u^n=u^n(\cdot)=u(\cdot,t_n)$ and $u_j^n=u(x_j,t_n)$.
For simplicity, we also introduce the following notations
$$\delta_x^2u_{j}^n=\frac{u^n_{j+1}-2u^n_j+u_{j-1}^n}{h^2},{\quad}
\delta_{\hat{x}}u_{j}^n=\frac{u^n_{j+1}-u_{j-1}^n}{2h}.$$

We discretize the time   of \eqref{eq:subeq} through the fractional linear multistep methods
(FLMMs) developed  by Lubich \cite{Lubich1986}.
The $p$th-order FLMMs for  $D^{-\beta}_{0,t}u(t)$ are given by
\begin{equation} \label{lubich1}
D^{-\beta}_{0,t}u(t)\big|_{t=t_n}
=\tau^{\beta}\sum_{k=0}^{n}\omega^{(\beta)}_{n-k}u(t_k)
+\tau^{\beta}\sum_{k=0}^{s}w^{(\beta)}_{n,k}u(t_k)+O(\tau^{p}),
\end{equation}
where $\{\omega^{(\beta)}_{k}\}$ can be the coefficients of the Taylor expansions of
the following generating functions
\begin{eqnarray}
&&w^{(\beta)}(z)=\left[\sum_{j=1}^p\frac{1}{j}(1-z)^j\right]^{-\beta},
{\quad}p=1,2,...,6,\label{eq:w1}\\
&&w^{(\beta)}(z)=(1-z)^{-\beta}{\left[\gamma_0+\gamma_1(1-z)
+\gamma_2(1-z)^2+...+\gamma_{p-1}(1-z)^{p-1}\right]},\label{eq:w2}\\
&&w^{(\beta)}(z)=\left(\frac{1}{2}\frac{1+z}{1-z}\right)^{\beta},\label{eq:w3}
\end{eqnarray}
in which $\{\gamma_k\}$ in
\eqref{eq:w2} satisfy the following relation
$$\left(\frac{\ln{z}}{z-1}\right)^{-\beta}=\sum_{k=0}^{\infty}\gamma_k(1-z)^k,
{\quad}\gamma_0=1, \gamma_1=-\frac{\beta}{2}.$$
The starting weights $\{w^{(\beta)}_{n,k}\}$  are chosen such
that the asymptotic behavior of the function $u(t)$ near the
origin ($t=0$) are taken into account \cite{DiethelmFordFreedWeilbeer2006}.
One way to determine  $\{w^{(\beta)}_{n,k}\}$ for the sufficiently smooth function $u(t)$ is
given as follows \cite{Lubich1986}
\begin{equation} \label{wj}
{\sum_{k=1}^{p}}\omega^{(\beta)}_{n,k}k^q=\frac{\Gamma(q+1)}{\Gamma(q+\beta+1)}n^{q+\beta}
-\sum_{k=1}^{n}\omega^{(\beta)}_{n-k}k^q, {\quad}q=0,1,\cdots,p-1.
\end{equation}

The FLMM \eqref{lubich1} (also called the fractional trapezoidal rule) has second-order accuracy
if the generating function \eqref{eq:w3}
is used. In  this section, we will discretize the time  of
the {fractional wave} equation \eqref{eq:subeq} with
the generating function \eqref{eq:w3}.

We first consider the   following fractional ordinary differential equation (FODE)
\begin{equation}\label{eq:FODE}
{}_{C}D^{\beta}_{0,t}y(t)={\mu}y(t)+g(t),{\quad}y(0)=y_0,y'(0)=y_0'{\quad}1<\beta<2.
\end{equation}
We also assume that $y(t)$ is sufficiently smooth. Let $\hat{\varphi}(t)=y(0)+y'(0)t$. Then the above
FODE is equivalent to the following Volterra integral equation {\cite{DieFF04}}
\begin{equation}\label{eq:VIE}
y(t)-\hat{\varphi}(t)
=\mu\,D^{-\beta}_{0,t}y(t) +D^{-\beta}_{0,t} g(t)=\mu\,D^{-\beta}_{0,t}(y(t)-\hat{\varphi}(t))
+\mu\,D^{-\beta}_{0,t} \hat{\varphi}(t)+D^{-\beta}_{0,t} g(t).
\end{equation}

{The kernel $\frac{1}{\Gamma(\beta)}(t-s)^{\beta-1}$ (see also Eq. \eqref{eq:fint}) in above
Eq. \eqref{eq:VIE} has no singularity for $\beta\geq1$. There exist several
difference methods to solve \eqref{eq:VIE} and the error estimates can be proved
by the generalized Gronwall inequality for any $\beta>0$, see e.g. \cite{CaoXu13,LiZeng13,LinLiu07}.
For $\beta\geq1$, the error estimate can be also proved by the classical Gronwall inequality
due to the nonsingularity of the kernel,
see e.g. \cite{LinLiu07}. Here, we use another way to discretize \eqref{eq:VIE} that will be used
to discretize the time of \eqref{eq:subeq}.}

Before discretizing  \eqref{eq:VIE}, we introduce three lemmas.
\begin{lemma}[\cite{Lubich1986,ZengLi2013a}]\label{lm3.1}
If $y(t)=t^{\nu},\nu\geq0,\beta>0$, then
\begin{equation}\label{eq:lm2-1}
\left[D^{-\beta}_{0,t}y(t)\right]_{t=t_n}
=\tau^{\beta}\sum_{k=0}^{n}\omega_{n-k}^{(\beta)}y(t_k)+O(t_n^{\nu+\beta-p}\tau^{p})
+O(t_n^{\beta-1}\tau^{\nu+1}),
\end{equation}
where $\{\omega^{(\beta)}_{k}\}$ can be the coefficients of the Taylor series of the generating functions
defined as \eqref{eq:w1}--\eqref{eq:w3}, and $p=2$ if  \eqref{eq:w3} is used.
\end{lemma}

\begin{lemma}[\cite{ZengLi2013a}]\label{lm3.2}
Denote by
\begin{equation}\label{eq:FLMM2}
y_n=\sum_{k=0}^n\omega_{n-k}^{(\beta)}G_k,
\end{equation}
{where $G_k\,(k=0,1,...)$ is any number} and
$\{\omega^{(\beta)}_{k}\}$ are the coefficients of Taylor expansions of
the generating functions $w^{(\beta)}(z)$ defined by Eq. \eqref{eq:w1}, Eq. \eqref{eq:w2},
or Eq. \eqref{eq:w3}.
Then, Eq. \eqref{eq:FLMM2} is  equivalent to the following form
\begin{equation}\label{eq:FLMM}
\sum_{k=0}^n\alpha_ky_{n-k}=\sum_{k=0}^n\theta_{n-k}G_k
\end{equation}
where  $\alpha_k$ and $\theta_k$ are the coefficients of Taylor expansions of $\alpha(z)$
and $\theta(z)$, respectively, with $w^{(\beta)}(z)=\theta(z)/\alpha(z)$.
\end{lemma}

\begin{lemma}[\cite{ZengLi2014}]\label{lm3.3}
Suppose that $\beta>0$. Let $\{\alpha_k\}$ be the coefficients of Taylor expansions of the generating function
$\alpha(z)=(1-z)^{\beta}$, i.e., $\alpha_k=(-1)^k\binom{\beta}{k}$. Then
$$\sum_{k=1}^n\alpha_{n-k}k^{\gamma-1}=O(n^{\gamma-1-\beta}) + O(n^{-\beta-1}),
{\quad}\gamma\in \mathbb{R},{\quad}\gamma\neq 0,-1,-2,\cdots.$$
\end{lemma}

Now, we are in a position to discretize \eqref{eq:VIE}.
If $y(t)$ is smooth enough, then we have
$y(t)-\hat{\varphi}(t)=\frac{1}{2}y''(0)t^2+D_{0,t}^{-3}y'''(t)$. Therefore,
by Lemma \ref{lm3.1}, we can have the following
discretization   for $\left[D^{-\beta}_{0,t}\left(y(t)-\hat{\varphi}(t)\right)\right]_{t=t_n}$ as
\begin{equation}\label{eq:VIE4}
\left[D^{-\beta}_{0,t}\left(y(t)-\hat{\varphi}(t)\right)\right]_{t=t_n}
=\tau^{\beta}\sum_{k=0}^n\omega_{n-k}^{(\beta)}(y(t_k)-\hat{\varphi}(t_k))- \widetilde{R}^n,
\end{equation}
where $\{\omega_{k}^{(\beta)}\}$  are the coefficients of Taylor expansions of
the  generating function \eqref{eq:w3}, and the truncation error  $\widetilde{R}^n$ satisfies
$\widetilde{R}^n=O(t_n^{\beta}\tau^2)+ O(t_n^{\beta-1}\tau^3) \textcolor[rgb]{1.00,0.00,0.00}{+}O(\tau^{2+\beta})
=O(n^{\beta}\tau^{2+\beta})+O(n^{\beta-1}\tau^{2+\beta})+O(\tau^{2+\beta})$,
where Lemma   \ref{lm3.1}  {with $\nu=p\,(p\geq 2)$
is used to obtain $\widetilde{R}^n$.}

Hence, Eq. \eqref{eq:VIE} has the following discretization
\begin{equation}\label{eq:VIE5}
y^n-\hat{\varphi}^n
={\mu}\tau^{\beta}\sum_{k=0}^n\omega_{n-k}^{(\beta)}(y^k-\hat{\varphi}^k)
+\mu\,\left[D^{-\beta}_{0,t}\hat{\varphi}(t)\right]_{t=t_n}+ \left[D^{-\beta}_{0,t}g(t)\right]_{t=t_n} +{\widetilde{R^n}}.
\end{equation}
Applying  Lemma \ref{lm3.2} yields the equivalent form of \eqref{eq:VIE5} as
\begin{equation}\label{eq:VIE6}\begin{aligned}
\sum_{k=0}^n\alpha_{n-k}\left(y^k-\hat{\varphi}^k\right)
=&{\mu\tau^{\beta}}\sum_{k=0}^n\theta_{n-k}\left(y^k-\hat{\varphi}^k\right)
+\sum_{k=0}^n{\alpha_{n-k}}\left\{\mu\,\left[D^{-\beta}_{0,t}\hat{\varphi}(t)\right]_{t=t_k}
+ \left[D^{-\beta}_{0,t}g(t)\right]_{t=t_k} +{\widetilde{R^k}}\right\},
\end{aligned}\end{equation}
where  $\alpha(z)$ and $\theta(z)$ in Lemma \ref{lm3.2}
can be chosen as  $\alpha(z)=(1-z)^{\beta}=\sum_{k=0}^{\infty}\alpha_kz^k
=\sum_{k=0}^{\infty}(-1)^k\binom{\beta}{k}z^k$ and
$\theta(z)=\frac{(1+z)^{\beta}}{2^{\beta}}=\sum_{k=0}^{\infty}\theta_kz^k
=\frac{1}{2^{\beta}}\sum_{k=0}^{\infty}\binom{\beta}{k}z^k$. One can also find that
$\theta_k=2^{-\beta}(-1)^k\alpha_k$.

Rewriting \eqref{eq:VIE6} into the following form
\begin{equation}\label{eq:VIE7}\begin{aligned}
\sum_{k=0}^n\alpha_{n-k}(y^{k}-\hat{\varphi}^k)
=&{\mu\tau^{\beta}}\sum_{k=0}^n\theta_{n-k}\left(y^k-\hat{\varphi}^k\right)
+\mu\,\sum_{k=0}^n\alpha_{k}\hat{\Phi}^k + \sum_{k=0}^n{\alpha_{n-k}}G^k +R^n,
\end{aligned}\end{equation}
where $G^k=\left[D^{-\beta}_{0,t}g(t)\right]_{t=t_k}$,
$\hat{\Phi}^k=\left[D^{-\beta}_{0,t}\hat{\varphi}(t)\right]_{t=t_k}=\frac{y_0t_k^{\beta}}{\Gamma(\beta+1)}
+\frac{y'_0t_k^{\beta+1}}{\Gamma(\beta+2)}$,
and $R^n=\sum_{k=1}^n\alpha_{n-k}{\widetilde{R^k}}$.

Next, we analyse the truncation error
$R^n=\sum_{k=0}^n{\alpha_{n-k}}\widetilde{R^k}$
defined in \eqref{eq:VIE7} when the
generating function \eqref{eq:w3}  is used.
We can obtain a bound of    $R^n$ in \eqref{eq:VIE7} as follows
\begin{equation}\label{eq:e1}
R^n=\sum_{k=0}^n\alpha_{n-k}{\widetilde{R^k}}
 =\sum_{k=0}^n\alpha_{n-k}\left(O(n^{\beta}\tau^{2+\beta})+O(n^{\beta-1}\tau^{2+\beta})+O(\tau^{2+\beta})\right)
 =O(\tau^{2+\beta}),
\end{equation}
where we have used  Lemma \ref{lm3.3}.

Assume  that $U(x,t)$ is sufficiently smooth in time. From  \eqref{eq:VIE7},
we can obtain the time discretization of the
wave equation \eqref{eq:subeq}   as follows.

\begin{itemize}
\item \textbf{Time discretization I:} Applying the time discretization \eqref{eq:VIE7} with  the generating
function \eqref{eq:w3} to Eq. \eqref{eq:subeq} yields
\begin{equation}\label{eq:t1}\begin{aligned}
\sum_{k=0}^n\alpha_{n-k}(U^{k}-\varphi^k)
=&{\mu\tau^{\beta}}\sum_{k=0}^n\theta_{n-k}\left(\px^2 U^k-\px^2\varphi^k\right)
+\mu\,\sum_{k=0}^n\alpha_{n-k}\px[x]^2\Phi^k + \sum_{k=0}^n\alpha_{n-k}F^k+R^n,
\end{aligned}\end{equation}
where $\alpha_k=(-1)^k\binom{\beta}{k},\theta_k=2^{-\beta}(-1)^k\alpha_k$,
$\varphi(x,t)=U(x,0)+{t\px[t]U(x,0)}=\phi_0(x)+\psi_0(x)t$,
$\Phi^k=\left[D^{-\beta}_{0,t}\varphi(x,t)\right]_{t=t_k}=\frac{\phi_0(x)t_k^{\beta}}{\Gamma(\beta+1)}
+\frac{\psi_0(x)t_k^{\beta+1}}{\Gamma(\beta+2)}$,
 $F^k=\left[D^{-\beta}_{0,t}f(x,t)\right]_{t=t_k}$,  and $R^n$ is the discretization error in time
satisfying $|R^n|{\,\leq\,}C\tau^{2+\beta}$.
\end{itemize}

Next, we present the fully discrete  approximation for  equation \eqref{eq:subeq}.
From the time discretization  \eqref{eq:t1}   with  the second-order central difference
discretization of the space derivative operator, we  present the
corresponding fully discrete approximations for \eqref{eq:subeq} as follows.
\begin{itemize}
\item \textbf{Scheme I:} Find $u_{j}^n$ for $j=1,2,...,N-1,n=1,2,...,n_T$, such that
\begin{equation}\label{fem1}
\left\{\begin{aligned}
&\sum_{k=0}^n\alpha_{n-k}(u_j^{k}-\varphi_j^k)
={\mu\tau^{\beta}}\sum_{k=0}^n\theta_{n-k}\left(\delta^2_xu_j^k-\delta^2_x\varphi_j^k\right)
+\mu\,\sum_{k=0}^n\alpha_{n-k}\delta_x^2\Phi^k_j + \sum_{k=0}^n\alpha_{n-k}F^k_j,\\
&u_0^k=U_a(t_k),{\quad}u_N^k=U_b(t_k),{\qquad}k=0,1,...,n_T.\\
&u_j^0=\phi_0(x_j),
\end{aligned}\right.
\end{equation}
where $\alpha_k=(-1)^k\binom{\beta}{k},\theta_k=2^{-\beta}(-1)^k\alpha_k$,
$\Phi^k_j=\left[D^{-\beta}_{0,t}\varphi(x_j,t)\right]_{t=t_k}=
\frac{\phi_0(x_j)t_k^{\beta}}{\Gamma(\beta+1)}
+\frac{\psi_0(x_j)t_k^{\beta+1}}{\Gamma(\beta+2)}$,
$\varphi_j^k=\phi_0(x_j)+\psi_0(x_j)t_k$,
and $F_j^k=\left[D^{-\beta}_{0,t}f(x_j,t)\right]_{t=t_k}$.
\end{itemize}

\begin{remark}
If $\beta \to 2$, then the  scheme \eqref{fem1}  reduces to the unconditionally stable
central difference scheme of second-order accuracy both in time and space, i.e.,
 \begin{equation*}\label{fdm2-2}
 \begin{aligned}
 &\frac{u_j^{n+1}-2u_j^n+u_j^{n-1}}{\tau^2}
 =\frac{\mu}{4}\left(\delta_x^2u_j^{n+1}+{2\delta_x^2u^{n}_j}+\delta_x^2u^{n-1}_j\right)
 +(F_j^{n+1}-2F_j^n+F^{n-1}_j),{\quad}n\geq 1.
 \end{aligned}
 \end{equation*}
\end{remark}
\textbf{Calculation of $F^n$:} In \eqref{fem1}, we do not illustrate how to
calculate $\left[D^{-\beta}_{0,t}f(x,t)\right]_{t=t_k}$ in $F^n$.
If $U(x,t)$ is sufficiently smooth in time, then $f(x,t)-f(x,0)$ has the form
$f(x,t)-f(x,0)=t^{2-\beta}f_1(x,t)+tf_2(x,t)$, where $f_1(x,t)$ and $f_2(x,t)$ are sufficiently smooth in time.
Hence, we can use the following second-order formula to approximate
$\left[D^{-\beta}_{0,t}f(x,t)\right]_{t=t_n}$
\begin{equation} \label{CalFn}\begin{aligned}
\left[D^{-\beta}_{0,t}f(x,t)\right]_{t=t_n}
=&\left[D^{-\beta}_{0,t}\left(f(x,t)-f(x,0)\right)\right]_{t=t_n} + \frac{t_n^{\beta}}{\Gamma(1+\beta)}f(x,0)\\
=&\tau^{\beta}\sum_{k=0}^{n}\omega^{(\beta)}_{n-k}(f(x,t_k)-f(x,0))
+\tau^{\beta}w^{(\beta)}_{n,1}(f(x,t_1)-f(x,0))\\
&{+\tau^{\beta}w^{(\beta)}_{n,2}(f(x,t_2)-f(x,0))+ \frac{t_n^{\beta}}{\Gamma(1+\beta)}f(x,0)+O(\tau^2),}
\end{aligned}\end{equation}
where $\{\omega^{(\beta)}_{k}\}$ are the coefficients of the Taylor expansions of
the  generating  function \eqref{eq:w2}.
The coefficients $\{w^{(\beta)}_{n,1}\}$ and $\{w^{(\beta)}_{n,2}\}$ are chosen such that
\eqref{CalFn} is exact for ${f(x,t)-f(x,0)=t^{2-\beta},t}$. Hence, one has
\begin{equation} \label{Calwj}
{w^{(\beta)}_{n,1}+2^q w^{(\beta)}_{n,2}=\frac{\Gamma(q+1)}{\Gamma(q+\beta+1)}n^{q+\beta}
-\sum_{k=1}^{n}\omega^{(\beta)}_{n-k}k^q,{\quad}q\in\{2-\beta,1\}.}
\end{equation}

\subsection{Stability, consistency,  and convergence}
This subsection mainly focuses on the stability, consistency, and convergence of  the  scheme \eqref{fem1}.
We first rewrite the scheme \eqref{fem1} into the following form
\begin{equation}
\begin{aligned}
&\sum_{k=0}^n\alpha_{n-k}(u_j^{k}-\varphi_j^k)
={\mu\tau^{\beta}}\sum_{k=0}^n\theta_{n-k}\delta^2_xu_j^k + \mu\tau^{\beta} H^n_j+ \sum_{k=0}^n\alpha_{n-k}F^k_j,
\end{aligned}
\end{equation}
where
\begin{equation}\label{hj}
\begin{aligned}
H_j^n=&-\sum_{k=0}^n\theta_{n-k}\delta^2_x\varphi_j^k
+{\tau^{-\beta}}\sum_{k=0}^n\alpha_{n-k}\delta_x^2\Phi^k_j
=\delta^2_x\phi_0(x_j)A_n+\tau\delta^2_x\psi_0(x_j)B_n,
\end{aligned}
\end{equation}
in which
\begin{equation}\label{ABn}
\begin{aligned}
&A_n=\frac{1}{\Gamma(\beta+1)}\sum_{k=0}^n\alpha_{n-k}k^{\beta}-\sum_{k=0}^n\theta_{n-k}\\
&B_n=\frac{1}{\Gamma(\beta+2)}\sum_{k=0}^n\alpha_{n-k}k^{\beta+1}-\sum_{k=0}^n\theta_{n-k}k.
\end{aligned}
\end{equation}

Consider \eqref{eq:lm2-1} with  $y(t)=t^{\nu},\nu\geq 0$, one has
\begin{equation}\label{fem1-0}
\frac{\Gamma(\nu+1)t_n^{\nu+\beta}}{\Gamma(\beta+\nu+1)}=\left[D^{-\beta}_{0,t}t^{\nu}\right]_{t=t_n}
=\tau^{\beta}\sum_{k=0}^n\omega_{n-k}^{(\beta)}(t_k)^{\nu} + O(t_{n}^{\nu+\beta-2}\tau^2)+O(t_{n}^{\beta-1}\tau^{\nu+1}).
\end{equation}
The above equation implies
\begin{equation}\label{fem1-0-1}
\frac{\Gamma(\nu+1)n^{\nu+\beta}}{\Gamma(\beta+\nu+1)}
=\sum_{k=0}^n\omega_{n-k}^{(\beta)}k^{\nu} + O({n}^{\nu+\beta-2})+O({n}^{\beta-1}).
\end{equation}
Applying Lemma  \ref{lm3.2} yields
\begin{equation}\label{fem1-0-2}
\frac{\Gamma(\nu+1)}{\Gamma(\beta+\nu+1)}\sum_{k=0}^n\alpha_{n-k} k^{\nu+\beta}
-\sum_{k=0}^n\theta_{n-k}k^{\nu} = \sum_{k=0}^n\alpha_{n-k}\left(O({k}^{\nu+\beta-2})+O({k}^{\beta-1})\right).
\end{equation}
By Lemma \ref{lm3.3} and  \eqref{fem1-0-2} with $\nu=0,1$, one has
\begin{equation}\label{fem1-0-3}
A_n = O({n}^{-1}),{\quad} B_n=O({n}^{-1}).
\end{equation}

Next, we prove the stability of the scheme \eqref{fem1}  through Fourier analysis.
\begin{theorem}\label{substblth}
The finite difference scheme \eqref{fem1} is   stable.
\end{theorem}
\begin{proof}
Suppose that $u_j^n=\rho^n\exp(ij\sigma h),i^2=-1$, $\psi_0(x_j)=d_0\exp(ij\sigma h)$, and $F_j^k=0$. Then we have
\begin{equation*}\label{fem1-1}\begin{aligned}
&\exp(ij\sigma h)\sum_{k=0}^n\alpha_{n-k}(\rho^k-\rho^0-d_0t_k)\\
=&\frac{\mu\tau^{\beta}}{h^2}\left[\sum_{k=0}^n\theta_{n-k}\rho^k+A_n\rho^0+B_nd_0\tau\right]
\Big(\exp(i(j+1)\sigma h)-2\exp(ij\sigma h)+\exp(i(j-1)\sigma h)\Big)\\
=&\exp(ij\sigma h)\left(-\frac{\mu\tau^{\beta}}{h^2}4\sin^2\left(\frac{\sigma h}{2}\right)\right)
\left[\sum_{k=0}^n\theta_{n-k}\rho^k+A_n\rho^0+B_nd_0\tau\right],
\end{aligned}\end{equation*}
where $\alpha_k=(-1)^k\binom{\beta}{k}$ and $\theta_k=2^{-\beta}(-1)^k\alpha_k$. Eliminating $\exp(ij\sigma h)$
from the above equation leads to
\begin{equation}\label{fem1-2}\begin{aligned}
\sum_{k=0}^n\alpha_{n-k}(\rho^k-\rho^0-d_0t_k)
=&S^*\left[\sum_{k=0}^n\theta_{n-k}\rho^k+A_n\rho^0+B_nd_0\tau\right],
\end{aligned}\end{equation}
where
$$S^*=-\frac{4\mu\tau^{\beta}}{h^2}\sin^2\left(\frac{\sigma h}{2}\right).$$
Now, we need to investigate the stability of the difference equation  \eqref{fem1-2}.
Let
$$A(z)=\sum_{k=0}^{\infty}A_kz^{k},{\quad}B(z)=\sum_{k=0}^{\infty}B_kz^{k},{\quad}
 \rho(z)=\sum_{k=0}^{\infty}\rho^kz^{k},\quad |z|\leq 1.$$
From \eqref{fem1-2}, one can obtain
\begin{equation*}\label{fem1-3}\begin{aligned}
\sum_{n=0}^{\infty}\left[\sum_{k=0}^n\alpha_{n-k}(\rho^k-\rho^0-d_0\tau k)\right]z^n
=&S^*\sum_{n=0}^{\infty}\left[A_n\rho^0+B_nd_0\tau+\sum_{k=0}^n\theta_{n-k}\rho^k\right]z^n,
\end{aligned}\end{equation*}
which implies
\begin{equation*}\label{fem1-4}\begin{aligned}
\alpha(z)\left(\rho(z)-\frac{\rho^0}{1-z}-d_0\tau K(z)\right)
=&S^*\left(\theta(z)\rho(z)+A(z)\rho^0+B(z)d_0\tau\right),
\end{aligned}\end{equation*}
where
\begin{equation}\label{fem1-kz}
K(z)=\sum_{k=0}^{\infty}kz^k=\frac{z}{(1-z)^2}.
\end{equation}
Hence,
\begin{equation}\label{fem1-5}\begin{aligned}
\rho(z)=\frac{\alpha(z)\left((1-z)^{-1}\rho^0+K(z)d_0\tau\right)-S^*\left(A(z)\rho^0+B(z)d_0\tau\right)}{\alpha(z)-S^*\theta(z)}.
\end{aligned}\end{equation}

Denote by $\alpha(z)\left((1-z)^{-1}\rho^0+d_0\tau K(z)\right)-S^*\left(A(z)\rho^0+B(z)d\tau\right)=\sum_{k=0}^{\infty}g_kz^k$.
From \eqref{fem1-0-3}, \eqref{fem1-kz}, \eqref{fem1-5}, and $\alpha(z)=(1-z)^{\beta}$, we can derive
that $g_n\to 0$ as $n\to \infty$ for any given  $S^*$ (see the first part in the proof of Lemma 3.5 in \cite{Lubich1986}
and Eq. (2.8) in \cite{Lubich1986b}),
which implies  $\rho^n\to 0$  as $n\to \infty$ for any given
$S^*\neq \frac{\alpha(z)}{\theta(z)},|z|\leq 1$ (see also Eq. (2.8) in \cite{Lubich1986b}). If there exists a number
$\sigma$ such that $\sin^2\left(\frac{\sigma h}{2}\right)=0$, i.e., $S^*=0$, then we can also
obtain $\rho^k=\rho^0+d t_k$ from \eqref{fem1-2}. Like  Theorem 2.1 in \cite{Lubich1986b},
we can obtain  the following region
$$\mathbb{S}=\mathbb{C}\setminus \left\{\frac{\alpha(z)}{\theta(z)}:|z|\leq 1\right\}$$
such that $\rho^n\,(0 \leq n\leq n_T)$ is bounded for any $\tau,h$ and given $T$.
That is to say, for any $S^*\in \mathbb{S}$, the difference equation \eqref{fem1-2} is stable.
Since $\alpha(z)=(1-z)^{\beta}\geq 0$ and $\theta(z)=2^{-\beta}(1+z)^{\beta} \geq0$ for all $z\in [-1,1]$,
the value of $\frac{\alpha(z)}{\theta(z)}$
is always nonnegative.
Therefore, the stability region $\mathbb{S}$ contains the whole of the left-half plane
(Of course, it contains the negative semi axis),
which implies that the difference relation \eqref{fem1-2} is stable for every $S^*<0$, i.e., $\rho^n$ is bounded as $n\to \infty$.
Hence, the scheme \eqref{fem1} is  stable for any $\tau^{\beta}/h^2$, which completes  the proof.
\end{proof}

Next, we investigate the consistency of the scheme \eqref{fem1}.
Letting $x=x_j$ in \eqref{eq:t1} and applying the central difference method to the space derivative, we can derive
\begin{equation}\label{consitancy-1}\begin{aligned}
\sum_{k=0}^n\alpha_{n-k}(U_j^{k}-\varphi_j^k)
=&{\mu\tau^{\beta}}\sum_{k=0}^n\theta_{n-k}\left(\delta_x^2 U_j^k-\delta_x^2 \varphi^k_j\right)
+\mu\,\sum_{k=0}^n\alpha_{n-k}\delta_x^2 \Phi_j^k + \sum_{k=0}^n\alpha_{n-k}F_j^k+R_j^n,
\end{aligned}\end{equation}
where $R_j^n=O(\tau^{\beta}(\tau^2+h^2))$.  In order to prove that the scheme \eqref{fem1} is consistent of order $O(\tau^2+h^2)$,
we just need to prove the following result
\begin{equation}\label{consitancy-2}\begin{aligned}
\lim_{\tau \to0,h\to 0}\Psi(\tau,h)={}_{C}D_{0,t}^{\beta}U(x_j,t)-\mu\,\px^2U(x_j,t)-f(x_j,t)=0, {\quad}n\tau=t,
\end{aligned}\end{equation}
where
\begin{equation}\label{consitancy-3}\begin{aligned}
\Psi(\tau,h) =& \frac{1}{\tau^{\beta}}\Bigg\{\sum_{k=0}^n\alpha_{n-k}(U_j^{k}-\varphi_j^k)
-\bigg[{\mu\tau^{\beta}}\sum_{k=0}^n\theta_{n-k}\left(\delta_x^2 U_j^k-\delta_x^2 \varphi^k_j\right)\\
&+\mu\,\sum_{k=0}^n\alpha_{n-k}\delta_x^2 \Phi_j^k + \sum_{k=0}^n\alpha_{n-k}F_j^k+R_j^n\bigg]\Bigg\}, {\quad}n\tau=t.
\end{aligned}\end{equation}
According to the Gr\"{u}nwald--Letnikov formula \cite{Podlubny1999}, we have
\begin{equation}\label{consitancy-4}\begin{aligned}
&\lim_{\tau \to0,h\to 0} \frac{1}{\tau^{\beta}}\sum_{k=0}^n\alpha_{n-k}(U_j^{k}-\varphi_j^k)
={}_{RL}D_{0,t}^{\beta}\left(U(x_j,t)-\varphi(x_j,t)\right)={}_{C}D_{0,t}^{\beta}U(x_j,t),\\
&\lim_{\tau \to0,h\to 0} \frac{1}{\tau^{\beta}}\sum_{k=0}^n\alpha_{n-k}F_j^k={}_{RL}D_{0,t}^{\beta}D_{0,t}^{-\beta}f(x_j,t)=f(x_j,t),
\end{aligned}\end{equation}
{where ${}_{RL}D_{0,t}^{\beta}$ is the Riemann--Liouville
fractional derivative operator, see e.g. \cite{Podlubny1999}.}

From \eqref{hj}, \eqref{ABn}, and \eqref{fem1-0-3}, we derive
\begin{equation}\label{consitancy-5}\begin{aligned}
&\lim_{\tau \to0,h\to 0}\frac{1}{\tau^{\beta}}\Bigg\{
{\tau^{\beta}}\sum_{k=0}^n\theta_{n-k}\delta_x^2 \varphi^k_j
-\sum_{k=0}^n\alpha_{n-k}\delta_x^2 \Phi_j^k \Bigg\}=0.
\end{aligned}\end{equation}
For $\sum_{k=0}^n\theta_{n-k}\delta_x^2 U_j^k$, we have
\begin{equation}\label{consitancy-6}\begin{aligned}
&\lim_{\tau \to0,h\to 0}\sum_{k=0}^n\theta_{n-k}\delta_x^2 U_j^k
=\lim_{\tau \to0}\sum_{k=0}^n\theta_{n-k}\px[x]^2U(x_j,t_k)\\
=&\lim_{\tau \to0}\sum_{k=0}^n\theta_{n-k}\px[x]^2\left(U(x_j,t)+(k-n)\tau\px[t]U(x_j,\xi({k,n}))\right){\quad}(0<\xi({k,n})<t)\\
=&\px[x]^2U(x_j,t)+\lim_{\tau \to0}\sum_{k=0}^n\theta_{n-k}(k-n)\tau\px[x]^2\px[t]U(x_j,\xi({k,n}))
=\px[x]^2U(x_j,t),
\end{aligned}\end{equation}
where we have used $\sum_{k=0}^n\theta_k\to 1$ as $n\to \infty$ and
\begin{equation}\label{consitancy-7}\begin{aligned}
&\lim_{\tau \to0}\Big|\sum_{k=0}^n\theta_{n-k}(k-n)\tau\px[x]^2\px[t]U(x_j,\xi({k,n}))\Big|{\quad} (t=n\tau)\\
\leq& C_1|t|\lim_{n \to\infty}\sum_{k=0}^n|\theta_{n-k}|\frac{n-k}{n}
\leq C_2|t|\lim_{n \to\infty}\sum_{k=1}^n(n-k)^{-\beta}n^{-1}=0,
\end{aligned}\end{equation}
in which  $\theta_k=2^{-\beta}\binom{\beta}{k}=O(k^{-\beta-1})$ and $\sum_{k=1}^n(n-k)^{-\beta} \leq C_3$ have been used,
$C_1,C_2,C_3$ are positive constants independent of $n,\tau$ and $h$.
Combining \eqref{consitancy-4}--\eqref{consitancy-7} yields \eqref{consitancy-2}.

We now give the following consistency theorem.
\begin{theorem}\label{thm-consist-1}
Suppose that $U(x,t)$ is the solution to \eqref{eq:subeq}, $U\in C^2(C^4(I);[0,T])$.
The finite difference scheme \eqref{fem1} is consistent of order $O(\tau^2+h^2)$.
\end{theorem}

Next, we discuss the convergence  for the scheme \eqref{fem1}. Denote by $e_j^n=U(x_j,t_n)-u_j^n$.
From \eqref{fem1} and \eqref{consitancy-1}, we obtain the error equation of \eqref{fem1} as follows
\begin{equation}\label{fem1-err}
\begin{aligned}
&\sum_{k=0}^n\alpha_{n-k}e_j^{k}
={\mu\tau^{\beta}}\sum_{k=0}^n\theta_{n-k}\delta^2_xe_j^k + R^n_j,
\end{aligned}
\end{equation}
where $R_j^n=O(\tau^{\beta}(\tau^2+h^2))=r_j^n\tau^{\beta}(\tau^2+h^2)$,  $r_j^n$ is bounded.

{Using the identity $R_j^n=\sum_{k=0}^n\alpha_{n-k}\sum_{l=0}^k\tilde{\alpha}_{k-l}R_j^l$
and Lemma \ref{lm3.2}, or applying the similar reasoning as Lemma 3.4 in  \cite{ZengLi2013a},}
we can derive the equivalent form of \eqref{fem1-err} as
\begin{equation}\label{fem1-err-1}
\begin{aligned}
&e_j^{n}={\mu\tau^{\beta}}\sum_{k=0}^n\omega^{(\beta)}_{n-k}\delta^2_xe_j^k + G_j^n,
\end{aligned}
\end{equation}
where $G_j^n=\sum_{k=0}^n\tilde{\alpha}_{n-k} R^k_j=O(\tau^2+h^2)$,
{$\tilde{\alpha}_{k}=(-1)^k\binom{-\beta}{k}$ is the coefficient
of the Taylor expansion of
the generating function $\tilde{\alpha}(z)=(\alpha(z))^{-1}=(1-z)^{-\beta}$},  $\omega^{(\beta)}_{k}$
is the coefficient of the Taylor expansion of the generating function $\theta(z)/\alpha(z)$.
Let $e^n_j=\epsilon^n\exp(ij\sigma h)$ and $G_j^n=\eta^n\exp(ij\sigma h)$, $\eta^n$ is bounded.
Similar to \eqref{fem1-2}, we can obtain from \eqref{fem1-err-1}
\begin{equation}\label{fem1-err-2}\begin{aligned}
\epsilon^n=S^*\sum_{k=0}^n\omega^{(\beta)}_{n-k}\epsilon^k + \eta^n(\tau^2+h^2),
{\quad}S^*=-\frac{4\mu\tau^{\beta}}{h^2}\sin^2\left(\frac{\sigma h}{2}\right),
\end{aligned}\end{equation}
which yields
\begin{equation}\label{fem1-err-3}\begin{aligned}
\epsilon(z)=S^*\epsilon(z)\theta(z)/\alpha(z) + \eta(z)(\tau^2+h^2),
\end{aligned}\end{equation}
where
$$\epsilon(z)=\sum_{k=0}^{\infty}\epsilon^kz^{k},{\quad}\eta(z)=\sum_{k=0}^{\infty}\eta^kz^{k}.$$
From \eqref{fem1-err-3}, we have
\begin{equation}\label{fem1-err-4}\begin{aligned}
\epsilon(z) =\frac{\eta(z)(\tau^2+h^2)}{1-S^*\theta(z)/\alpha(z)}
=\frac{(1-z)^{\beta}\eta(z)(\tau^2+h^2)}{(1-z)^{\beta}-S^*2^{-\beta}(1+z)^{\beta}}.
\end{aligned}\end{equation}
Denote by $D(z)=(1-z)^{\beta}-S^*2^{-\beta}(1+z)^{\beta}=\sum_{k=0}^nd_kz^k$. Then we have
$d_k=O(k^{-\beta-1}),1<\beta<2$. So the sequence $\{d_k\}$ is in $\ell^1$. From Theorem \ref{substblth}, we know that
$D(z)\neq 0$ for all $S^*\in \mathbb{S}$ and $|z|\leq 1$.   Denote by
$1/D(z)=\sum_{k=0}^n\hat{d}_kz^k$. Then the sequence $\{\hat{d}_k\}$ is also in $\ell^1$, see Eq. (2.5) in \cite{Lubich1986b}.
Let $(1-z)^{\beta}\eta(z)=\sum_{k=0}^nc_kz^k$. Then it is easy to obtain that
$|c_n|=|\sum_{k=0}^n\alpha_{k}\eta_{n-k}|\leq (\max_{0\leq k\leq n_T}{|\eta_k|})\sum_{k=0}^n|\alpha_{k}|
\leq 2^{\beta}\max_{0\leq k\leq n_T}{|\eta_k|}$.
From \eqref{fem1-err-4}, one has
\begin{equation}\label{fem1-err-5}
|\epsilon^n|=(\tau^2+h^2)\Big|\sum_{k=0}^n\hat{d}_kc_{n-k}\Big|
\leq  2^{\beta}\max_{0\leq k\leq n_T}{|\eta_k|}(\tau^2+h^2) \sum_{k=0}^n|\hat{d}_k|
\leq C(\tau^2+h^2).
\end{equation}
If $S^*=0$, then we directly have $|\epsilon^n|\leq C(\tau^2+h^2)$ from  \eqref{fem1-err-2}.
Now, we give the following convergence theorem.

\begin{theorem}\label{thm2-3}
Let $U(x,t)$ be the solution to \eqref{eq:subeq} and $u_j^n$ $(j=0,1,...,N,n=0,1,...,n_T)$ be the solutions
to \eqref{fem1}. Then there exists a positive constant $C$ independent of $n,\tau,$ and $h$, such that
$$\|e^n\|\leq C(\tau^2+h^2),$$
where $e^n=(e_0^n,e_1^n,...,e_N^n)^T,e_j^n=U(x_j,t_n)-u_j^n$, and
$\|e^n\|=\sqrt{h\sum_{j=0}^{N-1}(e_j^n)^2}.$
\end{theorem}
\begin{proof}
From $e^n_j=\epsilon^n\exp(ij\sigma h)$ and \eqref{fem1-err-5}, one has $|e_j^n|\leq C(\tau^2+h^2)$. So
$$\|e^n\|^2=h\sum_{j=0}^{N-1}(e_j^n)^2\leq C(\tau^2+h^2)^2,$$
which completes the proof.
\end{proof}

\section{The finite difference schemes based on the  generalized Newton-Gregory formula and its modification}
In this section, we construct another second-order difference scheme for  \eqref{eq:subeq} with the
help of the generating function
${w^{(\beta)}(z)=\frac{1-\frac{\beta}{2}+\frac{\beta}{2}z}{(1-z)^{\beta}}}$, see \eqref{eq:w2}
with $p=2$. Similar to \eqref{eq:t1},  we can derive the following time discretization with the help
of the generating function {$w^{(\beta)}(z)=\frac{1-\frac{\beta}{2}+\frac{\beta}{2}z}{(1-z)^{\beta}}$}.
\begin{equation}\label{eq:t2-0}\begin{aligned}
\sum_{k=0}^n\alpha_{n-k}(U^{k}-\varphi^k)
=&{\mu\tau^{\beta}}\sum_{k=0}^n\theta_{n-k}\left(\px^2 U^k-\px^2\varphi^k\right)
+\mu\,\sum_{k=0}^n\alpha_{n-k}\px[x]^2\Phi^k + \sum_{k=0}^n\alpha_{n-k}F^k+R^n,
\end{aligned}\end{equation}
where $\alpha_k=(-1)^k\binom{\beta}{k},\theta_0=1-\frac{\beta}{2},\theta_1=\frac{\beta}{2},\theta_k=0(k\geq 2)$,
$\varphi(x,t)=\phi_0(x)+\psi_0(x)t$,
$\Phi^k=\left[D^{-\beta}_{0,t}\varphi(x,t)\right]_{t=t_k}$,
 $F^k=\left[D^{-\beta}_{0,t}f(x,t)\right]_{t=t_k}$  and $R^n$ is the discretization error in time
satisfying $|R^n|{\,\leq\,}C\tau^{2+\beta}$.

From \eqref{eq:t2-0}, we can derive the  fully discrete finite difference scheme for \eqref{eq:subeq} as:  Find $u_{j}^n$ for $j=1,2,...,N-1,n=1,2,...,n_T$, such that
\begin{equation}\label{fem2-0-2}
\left\{\begin{aligned}
&\sum_{k=0}^n\alpha_{n-k}(u_j^{k}-\varphi_j^k)
={\mu\tau^{\beta}}\sum_{k=0}^n\theta_{n-k}\left(\delta^2_xu_j^k-\delta^2_x\varphi_j^k\right)
+\mu\,\sum_{k=0}^n\alpha_{n-k}\delta_x^2\Phi^k_j + \sum_{k=0}^n\alpha_{n-k}F^k_j,\\
&u_0^k=U_a(t_k),{\quad}u_N^k=U_b(t_k),{\qquad}k=0,1,...,n_T.\\
&u_j^0=\phi_0(x_j),
\end{aligned}\right.
\end{equation}
where $\alpha_k=(-1)^k\binom{\beta}{k},\theta_0=1-\frac{\beta}{2},\theta_1=\frac{\beta}{2},\theta_k=0(k\geq 2)$,
$\Phi^k_j=\left[D^{-\beta}_{0,t}\varphi(x_j,t)\right]_{t=t_k}=
\frac{\phi_0(x_j)t_k^{\beta}}{\Gamma(\beta+1)}
+\frac{\psi_0(x_j)t_k^{\beta+1}}{\Gamma(\beta+2)}$,
$\varphi_j^k=\phi_0(x_j)+\psi_0(x_j)t_k$,
and $F_j^k=\left[D^{-\beta}_{0,t}f(x_j,t)\right]_{t=t_k}$.

If $\beta\to 2$, then the method \eqref{fem2-0-2} 
a conditionally stable scheme $\frac{u_j^{n+1}-2u_j^n+u_j^{n-1}}{\tau^2}={\mu}\delta_x^2u_j^n+(F_j^{n+1}-2F_j^n+F^{n-1}_j)$
requiring $\frac{\tau^2}{h^2}\leq \frac{1}{\mu}$.
We may think that the method \eqref{fem2-0-2} is also conditionally stable. Similar to Theorem  \ref{substblth},
we can indeed obtain  the stability region for the method \eqref{fem2-0-2} below
$$\mathbb{S}=\mathbb{C}\setminus \left\{\frac{\alpha(z)}{\theta(z)}:|z|\leq 1\right\}
=\mathbb{C}\setminus \left\{\frac{(1-z)^{\beta}}{1-\frac{\beta}{2}+\frac{\beta}{2}z}:|z|\leq 1,1<\beta<2\right\}.$$
The above stability region contains the interval $[\frac{2^{\beta}}{1-\beta},0)$, which implies
$$\frac{2^{\beta}}{1-\beta}\leq-\frac{4\mu\tau^{\beta}}{h^2}\sin^2\left(\frac{\sigma h}{2}\right)<0.$$
From the above inequality, we can derive a CLF condition for the method \eqref{fem2-0-2}  as follows
\begin{equation}\label{scond}
\frac{2^{\beta}}{1-\beta}\leq-\frac{4\mu\tau^{\beta}}{h^2}<0,{\quad}\Longleftrightarrow{\quad}
0<\frac{\mu(\beta-1)\tau^{\beta}}{2^{\beta-2}h^2}\leq 1.
\end{equation}

Next, we make a slight modification of the scheme \eqref{fem2-0-2}  such that the derived scheme is stable for
any given real value of $\tau^{\beta}/h^2$.  We make a slight modification of the first term
$\sum_{k=0}^n\theta_{n-k}\left(\px^2 U^k-\px^2\varphi^k\right)$ in the right hand side of \eqref{eq:t2-0}
as follows
\begin{equation}\label{fem2-0-4}
\begin{aligned}
\sum_{k=0}^n\theta_{n-k}\left(\px^2 U^k-\px^2\varphi^k\right)
=&(1-\frac{\beta}{2})\left(\px^2 U^n-\px^2\varphi^n\right)
+\frac{\beta}{2}\left(\px^2 U^{n-1}-\px^2\varphi^{n-1}\right)\\
=&(1-\frac{\beta}{4})\left(\px^2 U^n-\px^2\varphi^n\right)
+\frac{\beta}{4}\left(\px^2 U^{n-2}-\px^2\varphi^{n-2}\right)+O(\tau^2).
\end{aligned}
\end{equation}
Combining \eqref{eq:t2-0} and \eqref{fem2-0-4}, we obtain the following new time discretization approach.
\begin{itemize}
\item \textbf{Time discretization II:}
\begin{equation}\label{eq:t2}\begin{aligned}
\sum_{k=0}^n\alpha_{n-k}(U^{k}-\varphi^k)
=&{\mu\tau^{\beta}}\sum_{k=0}^n\theta_{n-k}\left(\px^2 U^k-\px^2\varphi^k\right)
+\mu\,\sum_{k=0}^n\alpha_{n-k}\px^2\Phi^k + \sum_{k=0}^n\alpha_{n-k}F^k+R^n,
\end{aligned}\end{equation}
where $\alpha_k=(-1)^k\binom{\beta}{k},\theta_0=1-\frac{\beta}{4},\theta_1=0,\theta_2=\frac{\beta}{4},\theta_k=0(k\geq 3)$,
$\varphi(x,t)=U(x,0)+\px[t]U(x,0)t=\phi_0(x)+\psi_0(x)t$,
$\Phi^k=\left[D^{-\beta}_{0,t}\varphi(x,t)\right]_{t=t_k}=\frac{\phi_0(x)t_k^{\beta}}{\Gamma(\beta+1)}
+\frac{\psi_0(x)t_k^{\beta+1}}{\Gamma(\beta+2)}$,
 $F^k=\left[D^{-\beta}_{0,t}f(x,t)\right]_{t=t_k}$  and $R^n$ is the discretization error in time
satisfying $|R^n|{\,\leq\,}C\tau^{2+\beta}$.
\end{itemize}

From \eqref{eq:t2}, we can derive the following fully discrete finite difference scheme.
\begin{itemize}
\item \textbf{Scheme II:} Find $u_{j}^n$ for $j=1,2,...,N-1,n=1,2,...,n_T$, such that
\begin{equation}\label{fdm2}
\left\{\begin{aligned}
&\sum_{k=0}^n\alpha_{n-k}(u_j^{k}-\varphi_j^k)
={\mu\tau^{\beta}}\sum_{k=0}^n\theta_{n-k}\left(\delta^2_xu_j^k-\delta^2_x\varphi_j^k\right)
+\mu\,\sum_{k=0}^n\alpha_{n-k}\delta_x^2\Phi^k_j + \sum_{k=0}^n\alpha_{n-k}F^k_j,\\
&u_0^k=U_a(t_k),{\quad}u_N^k=U_b(t_k),{\qquad}k=0,1,...,n_T.\\
&u_j^0=\phi_0(x_j),
\end{aligned}\right.
\end{equation}
where $\alpha_k=(-1)^k\binom{\beta}{k},\theta_0=1-\frac{\beta}{4},\theta_1=0,\theta_2=\frac{\beta}{4},\theta_k=0(k\geq 3)$,
$\Phi^k_j=\left[D^{-\beta}_{0,t}\varphi(x_j,t)\right]_{t=t_k}=
\frac{\phi_0(x_j)t_k^{\beta}}{\Gamma(\beta+1)}
+\frac{\psi_0(x_j)t_k^{\beta+1}}{\Gamma(\beta+2)}$,
$\varphi_j^k=\phi_0(x_j)+\psi_0(x_j)t_k$,
and $F_j^k=\left[D^{-\beta}_{0,t}f(x_j,t)\right]_{t=t_k}$.
\end{itemize}

\begin{remark}
If $\beta \to 2$, then the  scheme \eqref{fdm2}  is reduced to  the following unconditionally stable scheme
 \begin{equation*}\label{fdm2-3}
 \begin{aligned}
 &\frac{u_j^{n+1}-2u_j^n+u_j^{n-1}}{\tau^2}
 =\frac{\mu}{2}\left(\delta_x^2u_j^{n+1}+\delta_x^2u^{n-1}_j\right)+(F_j^{n+1}-2F_j^n+F^{n-1}_j),{\quad}n\geq 1.
 \end{aligned}
 \end{equation*}
\end{remark}

Similar to Theorem \ref{substblth}, we can prove that the scheme \eqref{fdm2} is stable, we just need
to replace $\theta(z)=2^{-\beta}(1+z)^{\beta}$ in \eqref{fem1} with $\theta(z)=1-\frac{\beta}{4}+\frac{\beta}{4}z^2$
to get the desired result.
Like Theorems \ref{thm-consist-1} and \ref{thm2-3}, we can easily  prove that  the scheme \eqref{fdm2} is consistent
and convergent  of order $O(\tau^2+h^2)$.

\begin{remark}
In fact, we use a new generating function   $w^{(\beta)}(z)=\frac{1-\frac{\beta}{4}+\frac{\beta}{4}z^2}{(1-z)^{\beta}}$
in the construction of the scheme \eqref{fdm2}.
\end{remark}

%
%

\section{Fractional diffusion-wave with linear advection-reaction term}
In this section, we extend the time discretization techniques used in \eqref{fem1} and \eqref{fdm2} to the following equation
\begin{equation}\label{eq:wave2}
\left\{\begin{aligned}
&{}_{C}D_{0,t}^{\beta}U(x,t)+K_1U(x,t)+K_2\px[x]U(x,t)=\mu\,\px^2U(x,t)+f(x,t),\\
&{\qquad\qquad\qquad\qquad}(x,t){\,\in\,}I{\times}(0,T],I=(a,b),T>0,\\
&U(x,0)=\phi_0(x),{\quad}\px[t]U(x,0)=\psi_0(x),{\quad}x{\,\in\,}I,\\
&U(a,t)=U_a(t),{\quad}U(b,t)=U_b(t),{\quad}t{\,\in\,}(0,T],
\end{aligned}\right.
\end{equation}
where $1<\beta<2,\,\mu>0,K_1,K_2\geq 0$. See e.g. \cite{ChenLiu-etal2012} for  the case of $K_2=0$.

The time in \eqref{eq:wave2} is discretized similarly to the technique used in \eqref{fem1} or \eqref{fdm2},
the first-order and second-order space derivative operators are both discretized by the central difference method,
we directly give the fully scheme for   \eqref{eq:wave2} as follows.
\begin{itemize}
\item \textbf{Scheme III ($m$):} Find $u_{j}^n$ for $j=1,2,...,N-1,n=0,1,2,...,n_T-1$, such that
\begin{equation}\label{fdm-e1}
\left\{\begin{aligned}
&\sum_{k=0}^n\alpha_{n-k}(u_j^{k}-\varphi_j^k)
={\mu\tau^{\beta}}\sum_{k=0}^n\theta_{n-k}\left(\delta^2_xu_j^k-\delta^2_x\varphi_j^k\right)
+\mu\,\sum_{k=0}^n\alpha_{n-k}\delta_x^2\Phi^k_j + \sum_{k=0}^n\alpha_{n-k}F^k_j,\\
&{\qquad\qquad\qquad}-{K_2\tau^{\beta}}\sum_{k=0}^n\theta_{n-k}\left(\delta_{\hat{x}}u_j^k-\delta_{\hat{x}}\varphi_j^k\right)
-K_2\,\sum_{k=0}^n\alpha_{n-k}\delta_{\hat{x}}\Phi^k_j \\
&{\qquad\qquad\qquad}-{K_1\tau^{\beta}}\sum_{k=0}^n\theta_{n-k}\left(u_j^k-\varphi_j^k\right)
-K_1\,\sum_{k=0}^n\alpha_{n-k}\Phi^k_j, \\
&u_0^k=U_a(t_k),{\quad}u_N^k=U_b(t_k),{\qquad}k=0,1,...,n_T.\\
&u_j^0=\phi_0(x_j),
\end{aligned}\right.
\end{equation}
where $\alpha_k=(-1)^k\binom{\beta}{k}$,
$\Phi^k_j=\left[D^{-\beta}_{0,t}\varphi(x_j,t)\right]_{t=t_k}=
\frac{\phi_0(x_j)t_k^{\beta}}{\Gamma(\beta+1)}
+\frac{\psi_0(x_j)t_k^{\beta+1}}{\Gamma(\beta+2)}$,
$\varphi_j^k=\phi_0(x_j)+\psi_0(x_j)t_k$,
 $F_j^k=\left[D^{-\beta}_{0,t}f(x_j,t)\right]_{t=t_k}$,
and
\begin{equation}
m=\left\{\begin{aligned}
&1,{\qquad}\theta_k=\frac{(-1)^k}{2^{\beta}}\alpha_k,k=0,1,...;\\
&2,{\qquad}\theta_0=1-\frac{\beta}{4},\,\theta_1=0,\,\theta_2=\frac{\beta}{4},\,\theta_k=0\,(k\geq 3).
\end{aligned}\right.
\end{equation}
\end{itemize}

{Similar to Theorem \ref{substblth},
the finite difference method \eqref{fdm-e1} can be proven to be
stable, we just need to
replace $S^*$ in the proof of Theorem \ref{substblth} with
$$S^*=-\frac{4\mu\tau^{\beta}}{h^2}\sin^2\left(\frac{\sigma h}{2}\right)
-K_1\tau^{\beta}-i\frac{K_2\tau^{\beta}}{h}\sin(\sigma h)$$
to reach the conclusion. The consistency  of order $O(\tau^2+h^2)$ of  \eqref{fdm-e1}
can be also similarly proved as that of Theorem \ref{thm-consist-1}.
The stability and convergence rate are also shown   numerically in
the following section.}

\section{Numerical examples}\label{sec4}
In this section, we present numerical examples to verify the theoretical analysis in the previous sections.
We first numerically verify the error estimates and the convergence orders of
Scheme I (see Eq. \eqref{fem1}), the scheme \eqref{fem2-0-2},  and Scheme II (see Eq. \eqref{fdm2}).

\begin{example}\label{eg1}
Consider the following  diffusion-wave equation \cite{Sun06}
\begin{equation}\label{eq:eg1}
\left\{\begin{aligned}
&{}_{C}D_{0,t}^{\beta}U(x,t)=\px[x]^2U(x,t)+f(x,t),{\quad}(x,t){\,\in\,}(0,1){\times}(0,1],\\
&U(x,0)=2\exp(x),{\quad}\px[t]U(x,0)=\exp(x),{\quad}x{\,\in\,}(0,1),\\
&U(0,t)=t^{2+\beta}+t^2+t+2, {\quad}U(1,t)=(t^{2+\beta}+t^2+t+2)\exp(1){\quad}t\in(0,1],
\end{aligned}\right.
\end{equation}
Choose a suitable right hand side function $f(x,t)$ such that the exact solution  to  \eqref{eq:eg1} is
$$U(x,t)=(t^{2+\beta}+t^2+t+2)\exp(x).$$
\end{example}

Denote  $e_j^n=e_j^n(\tau,h)=U(x_j,t_n)-u_j^n$ as the error equation at time level $n$.
The convergence orders in time and space in the sense of the $L^2$ norm are defined as
\begin{equation}\label{eq:order}
\text{order}=\left\{
\begin{aligned}
&{\log(\|e^n(\tau_1,h)\|/\|e^n(\tau_2,h)\|)}/{\log(\tau_1/\tau_2)},{\quad} \text{in time},\\
&{\log(\|e^n(\tau,h_1)\|/\|e^n(\tau,h_2)\|)}/{\log(h_1/h_2)},{\quad} \text{in space},
\end{aligned}\right.
\end{equation}
where $\tau,\tau_1,\tau_2\,(\tau_1\neq\tau_2)$ and
$h,h_1,h_2\,(h_1{\neq}h_2)$ are the time and space step sizes, respectively, and
$$\|e^n\|=\left(h\sum_{j=0}^{N-1}(U(x_j,t_n)-u_j^n)^2\right)^{1/2}.$$

We first check the accuracy of the schemes \eqref{fem1} and \eqref{fdm2} in time,   and
the space and time steps sizes are
chosen as $h=1/1000$ and $\tau=1/16,1/32,1/64,1/128,1/256$,  the    $L^2$ error $\max_{0{\leq}n{\leq}n_T}\|e^n\|$
is shown in Table \ref{tb1-1}. It is found that   Scheme I \eqref{fem1} and   II  \eqref{fdm2} both show second-order
accuracy in time for different fractional order $\beta\,(\beta=1.1,1.5,1.9)$, which is inline with the theoretical analysis.
In Table \ref{tb1-2}, we show the convergence rates in space for the two schemes \eqref{fem1} and \eqref{fdm2}, from which
the second-order accuracy is observed.

We also test the accuracy and stability of the method \eqref{fem2-0-2}. From \eqref{scond}, one knows that the method
\eqref{fem2-0-2} is stable if $0<r=\frac{\mu(\beta-1)\tau^{\beta}}{2^{\beta-2}h^2}\leq1$.
We use the method \eqref{fem2-0-2} to solve \eqref{eq:eg1}, the numerical results are shown in  Table \ref{tb1-3},
which shows that the method  \eqref{fem2-0-2} is stable for $r\leq 1$, and unstable for $r>1$
(see stars * in Table \ref{tb1-3}, which means the numerical solutions blow up when $r>1$).  The numerical result is inline with
the theoretical result \eqref{scond}. The numerical results
in Table \ref{tb1-3} also show second-order accuracy both in time and space by the simple calculation
using \eqref{eq:order}.

Here, we also compare Scheme I and Scheme II with the finite difference scheme developed in \cite{Sun06}
with convergence of order $O(\tau^{3-\beta}+h^2)$, the results are shown in Table \ref{tb1-4}.
Obviously, the present methods show better performances because of their high-order convergence in time,
especially when $\beta$ tends to $2$.

\begin{table}
\caption{{The    $L^{2}$ errors $\max\limits_{0{\leq}n{\leq}n_T}\|e^n\|$ for Example \ref{eg1}, $N=1000$.}}\label{tb1-1}
\begin{tabular}{llllllll}
\hline\noalign{\smallskip}
Methods & $1/\tau$ & $\beta=1.1$ & order& $\beta=1.5$ & order& $\beta=1.9$ & order \\
\noalign{\smallskip}\hline\noalign{\smallskip}
            &16 &3.8171e-4&      &8.3180e-4&      &8.5199e-4&       \\
            &32 &9.6497e-5&1.9839&2.0297e-4&2.0350&1.5672e-4&2.4427\\
Scheme I    &64 &2.4292e-5&1.9900&5.0160e-5&2.0167&3.2382e-5&2.2749\\
\eqref{fem1}&128&6.1308e-6&1.9863&1.2504e-5&2.0042&7.2483e-6&2.1595\\
            &256&1.5774e-6&1.9585&3.1593e-6&1.9847&1.6813e-6&2.1081\\
 \hline
            &16 &1.5850e-3&      &2.4395e-3&      &2.0188e-3&       \\
            &32 &4.0136e-4&1.9815&6.0870e-4&2.0028&5.0043e-4&2.0122 \\
Scheme II   &64 &1.0101e-4&1.9904&1.5214e-4&2.0004&1.2407e-4&2.0120 \\
\eqref{fdm2}&128&2.5374e-5&1.9931&3.8069e-5&1.9987&3.0911e-5&2.0050 \\
            &256&6.3960e-6&1.9881&9.5594e-6&1.9936&7.7600e-6&1.9940 \\
\noalign{\smallskip}\hline
\end{tabular}
\end{table}

\begin{table}
\caption{{The   $L^{2}$ errors $\max\limits_{0{\leq}n{\leq}n_T}\|e^n\|$ for Example \ref{eg1},
$\tau=5\times 10^{-4}$.}}\label{tb1-2}
\begin{tabular}{llllllll}
\hline\noalign{\smallskip}
Methods & $N$ & $\beta=1.1$ & order& $\beta=1.5$ & order& $\beta=1.9$ & order \\
\hline\noalign{\smallskip}
            &16 &2.2415e-4&      &2.2608e-4&      &2.7949e-4&       \\
            &32 &5.6079e-5&1.9989&5.6560e-5&1.9989&6.9901e-5&1.9994\\
Scheme I    &64 &1.4040e-5&1.9979&1.4178e-5&1.9961&1.7486e-5&1.9991\\
\eqref{fem1}&128&3.5287e-6&1.9923&3.5826e-6&1.9846&4.3816e-6&1.9967\\
            &256&9.0088e-7&1.9697&9.3364e-7&1.9401&1.1050e-6&1.9875\\
 \hline
            &16 &2.2423e-4&      &2.2618e-4&      &2.7960e-4&       \\
            &32 &5.6158e-5&1.9974&5.6665e-5&1.9969&7.0014e-5&1.9977 \\
Scheme II   &64 &1.4119e-5&1.9919&1.4283e-5&1.9881&1.7599e-5&1.9922 \\
\eqref{fdm2}&128&3.6078e-6&1.9684&3.6876e-6&1.9536&4.4937e-6&1.9695 \\
            &256&9.7994e-7&1.8803&1.0386e-6&1.8280&1.2177e-6&1.8838 \\
\hline\noalign{\smallskip}
\end{tabular}
\end{table}

\begin{table}
\caption{{The  $L^{2}$ errors $\max\limits_{0{\leq}n{\leq}n_T}\|e^n\|$ for Example \ref{eg1} with method \eqref{fem2-0-2},
$r=\frac{\tau^{\beta}(\beta-1)}{h^22^{\beta-2}}$.}}\label{tb1-3}
\begin{tabular}{llllllll}
\hline\noalign{\smallskip}
$1/\tau$ & $1/h$ & $\beta=1.1$ & $r$ & $\beta=1.5$ & $r$& $\beta=1.9$ & $r$ \\
\hline\noalign{\smallskip}
200 &16 &2.2620e-4& 0.1406&2.2762e-4& 0.0640&2.7163e-4&0.0105 \\
200 &32 &5.8128e-5& 0.5625&5.8100e-5& 0.2560&6.2044e-5&0.0419 \\
200 &64 &3.5957e-2& 2.2499&1.5717e-5& 1.0240&9.6312e-6&0.1678 \\
200 &128&*        & 8.9994&*        & 4.0960&4.2557e-6&0.6711 \\
200 &256& *       &35.9976&*        &16.3840&*        &2.6845 \\
 \hline
1000&16 &2.2421e-4&0.0239&2.2609e-4&0.0057&2.7916e-4&0.0005 \\
1000&32 &5.6137e-5&0.0958&5.6573e-5&0.0229&6.9573e-5&0.0020 \\
1000&64 &1.4098e-5&0.3831&1.4190e-5&0.0916&1.7158e-5&0.0079 \\
1000&128&*        &1.5323&3.5947e-6&0.3664&4.0529e-6&0.0315 \\
1000&256& *       &6.1292&*        &1.4654&7.7662e-7&0.1261 \\
 \hline
2000&16 &2.2415e-4&0.0112&2.2604e-4&0.0020&2.7940e-4&0.0001 \\
2000&32 &5.6075e-5&0.0447&5.6525e-5&0.0081&6.9809e-5&0.0005 \\
2000&64 &1.4036e-5&0.1787&1.4143e-5&0.0324&1.7394e-5&0.0021 \\
2000&128&3.5246e-6&0.7148&3.5476e-6&0.1295&4.2891e-6&0.0084 \\
2000&256&*        &2.8594&8.9864e-7&0.5181&1.0133e-6&0.0338 \\
\hline\noalign{\smallskip}
\end{tabular}
\end{table}

\begin{table}
\caption{{Comparison of the    $L^{2}$ errors $\max\limits_{0{\leq}n{\leq}n_T}\|e^n\|$ of different methods, $N=1000$.}}\label{tb1-4}
\begin{tabular}{llllllll}
\hline\noalign{\smallskip}
Methods & $1/\tau$ & $\beta=1.1$ &$\beta=1.3$& $\beta=1.5$ &$\beta=1.65$& $\beta=1.8$ &  $\beta=1.95$ \\
\hline\noalign{\smallskip}
                &16 &6.0565e-4&8.1295e-4&1.0579e-3&1.2092e-3&1.0682e-3&1.0472e-3 \\
                &32 &1.5249e-4&2.0324e-4&2.5943e-4&2.9083e-4&2.5239e-4&1.9743e-4\\
Scheme I        &64 &3.8249e-5&5.0792e-5&6.4231e-5&7.1239e-5&6.1129e-5&4.1599e-5\\
\eqref{fem1}    &128&9.5771e-6&1.2694e-5&1.5978e-5&1.7625e-5&1.5031e-5&9.4623e-6\\
                &256&2.3960e-6&3.1726e-6&3.9844e-6&4.3830e-6&3.7264e-6&2.2504e-6\\
 \hline
                &16 &1.8083e-3&2.2323e-3&2.6654e-3&2.9169e-3&2.8089e-3&1.8335e-3 \\
                &32 &4.5731e-4&5.6274e-4&6.6515e-4&7.2234e-4&6.9427e-4&4.6041e-4 \\
Scheme II       &64 &1.1497e-4&1.4126e-4&1.6621e-4&1.7965e-4&1.7216e-4&1.1465e-4 \\
\eqref{fdm2}    &128&2.8820e-5&3.5385e-5&4.1543e-5&4.4793e-5&4.2845e-5&2.8574e-5 \\
                &256&7.2146e-6&8.8547e-6&1.0385e-5&1.1183e-5&1.0686e-5&7.1308e-6 \\
 \hline
                &16 &8.2488e-4&2.4537e-3&8.0555e-3&1.9355e-2&4.3571e-2&9.0347e-2\\
                &32 &2.1707e-4&7.4231e-4&2.8287e-3&7.6281e-3&1.9237e-2&4.4283e-2\\
Method\cite{Sun06} &64 &5.7161e-5&2.2527e-4&9.9490e-4&2.9957e-3&8.4268e-3&2.1535e-2\\
                &128&1.5065e-5&6.8536e-5&3.5040e-4&1.1750e-3&3.6782e-3&1.0435e-2\\
                &256&3.9737e-6&2.0894e-5&1.2354e-4&4.6078e-4&1.6029e-3&5.0475e-3\\
\hline\noalign{\smallskip}
\end{tabular}
\end{table}

\begin{example}\label{eg2}
{Consider the following  equation
\begin{equation}\label{eq:eg2}
\left\{\begin{aligned}
&{}_{C}D_{0,t}^{\beta}U(x,t)+U(x,t)+\px[x]U(x,t)=\px[x]^2U(x,t)+f(x,t),{\quad}(x,t){\,\in\,}(0,1){\times}(0,1],\\
&U(x,0)=1,{\quad}\px[t]U(x,0)=-x,{\quad}x{\,\in\,}(0,1),\\
&U(0,t)=1, {\quad}U(1,t)= \exp(-t){\quad}t\in(0,1],
\end{aligned}\right.
\end{equation}
where $1<\beta<2$. Choose the suitable $f(x,t)$ satisfies
$$f(x,t)=x^2t^{2-\beta}\sum_{k=0}^{\infty}\frac{(-xt)^k}{\Gamma(k+3-\beta)}
+\exp(-xt)-t\exp(-xt)-t^2\exp(-xt)$$
such that \eqref{eq:eg2} has the following analytical solution
$$U(x,t)=\exp(-xt).$$}
\end{example}

In this example, we test the convergence rates of Scheme III (1) and Scheme III (2), and we also compare the present
methods with the existing time discretization used in \cite{DingLi2013}, see also \cite{JinLazarovZhou2014},
where the second-order fractional backward difference formula
was used to discretize the Caputo derivative operator. We choose the time step size
and space step size as $\tau=h$,
the  $L^2$ errors for different fractional order $\beta$ are shown in Table  \ref{tb2-1}.
Clearly, the two methods of the present paper show second-order
accuracy both in time and space, while the second-order method in \cite{DingLi2013,JinLazarovZhou2014}
do not show second-order accuracy, especially when $\beta$ tends to 2, the method
in \cite{DingLi2013} is   reduced to
\begin{equation*}\label{cls-4}
\begin{aligned}
{\frac{1}{\tau^2}
\left({\frac{9}{4}u_j^{n+1}-6u_j^n+\frac{11}{2}u_j^{n-1}-2u_{j}^{n-2}+\frac{1}{4}u_j^{n-3}}\right)
={\mu}\delta_x^2u_j^{n+1}+f^{n+1}_j},
\end{aligned}
\end{equation*}
where $u_j^k$ $(k=1,2,3)$ 
should be derived with any known high-order methods,
or the second-order accuracy will possibly be lost.

\begin{table}
\caption{{Comparison of the  $L^{2}$ errors  at $t=1$
for Example \ref{eg2}, $\tau=h=1/N$.}}\label{tb2-1}
\begin{tabular}{cccccccc}
\hline\noalign{\smallskip}
$\beta $ &  $N$ &Scheme III (1) & order &  Scheme III (2)  & order &Method \cite{DingLi2013}  & order   \\
\hline\noalign{\smallskip}
       &32 &8.4650e-6&      &6.7927e-6&      &5.7930e-6&       \\
       &64 &2.1115e-6&2.0033&1.6853e-6&2.0110&1.4888e-6&1.9602 \\
$1.2$  &128&5.2733e-7&2.0015&4.1958e-7&2.0060&3.8205e-7&1.9623 \\
       &256&1.3178e-7&2.0006&1.0467e-7&2.0030&9.8097e-8&1.9615 \\
       &512&3.2939e-8&2.0003&2.6141e-8&2.0015&2.5239e-8&1.9586 \\
\hline
       &32 &6.2643e-6&      &6.8147e-6&      &3.2195e-5&       \\
       &64 &1.6018e-6&1.9675&1.6981e-6&2.0047&1.1534e-5&1.4809 \\
$1.5$  &128&4.0577e-7&1.9809&4.2395e-7&2.0020&4.0675e-6&1.5037 \\
       &256&1.0225e-7&1.9885&1.0592e-7&2.0010&1.4322e-6&1.5059 \\
       &512&2.5688e-8&1.9930&2.6469e-8&2.0005&5.0467e-7&1.5049 \\
\hline
       &32 &4.8058e-6&      &8.0483e-6&      &1.1786e-4&      \\
       &64 &1.2490e-6&1.9440&2.0105e-6&2.0011&5.5140e-5&1.0959\\
$1.7$  &128&3.2035e-7&1.9631&5.0129e-7&2.0039&2.3829e-5&1.2104\\
       &256&8.1431e-8&1.9760&1.2508e-7&2.0028&9.9913e-6&1.2540\\
       &512&2.0586e-8&1.9839&3.1233e-8&2.0017&4.1302e-6&1.2745\\
\hline
       &32 &4.8725e-6&      &7.1377e-6&      &1.5921e-4&       \\
       &64 &1.2457e-6&1.9676&1.8058e-6&1.9828&8.5657e-5&0.8943 \\
$1.8$  &128&3.1645e-7&1.9769&4.5167e-7&1.9993&4.0687e-5&1.0740 \\
       &256&7.9953e-8&1.9848&1.1286e-7&2.0008&1.8478e-5&1.1388 \\
       &512&2.0129e-8&1.9898&2.8213e-8&2.0000&8.2238e-6&1.1679 \\
\hline
       &32 &5.5981e-6&      &6.8122e-6&      &1.4257e-4&       \\
       &64 &1.4231e-6&1.9759&1.7064e-6&1.9972&9.7753e-5&0.5445 \\
$1.9$  &128&3.5907e-7&1.9867&4.2691e-7&1.9990&5.2516e-5&0.8964 \\
       &256&9.0096e-8&1.9947&1.0665e-7&2.0011&2.6086e-5&1.0095 \\
       &512&2.2531e-8&1.9996&2.6620e-8&2.0023&1.2548e-5&1.0558 \\
\hline
       &32 &7.9469e-6&      &1.1951e-5&      &1.1319e-4&      \\
       &64 &2.0114e-6&1.9822&2.9856e-6&2.0010&6.0261e-5&0.9094\\
$1.99$ &128&5.0676e-7&1.9888&7.4851e-7&1.9959&3.6628e-5&0.7183\\
       &256&1.2720e-7&1.9942&1.8752e-7&1.9970&2.0345e-5&0.8483\\
       &512&3.1860e-8&1.9972&4.6938e-8&1.9982&1.0681e-5&0.9296\\
       \hline
       &32 &9.2889e-6&      &1.3606e-5&      &1.3064e-4&      \\
       &64 &2.2342e-6&2.0558&3.3105e-6&2.0391&5.9485e-5&1.1350\\
$2$    &128&5.4612e-7&2.0324&8.1441e-7&2.0232&3.4066e-5&0.8042\\
       &256&1.3491e-7&2.0172&2.0190e-7&2.0121&1.8740e-5&0.8622\\
       &512&3.3520e-8&2.0090&5.0258e-8&2.0062&9.8582e-6&0.9267\\
\hline\noalign{\smallskip}
\end{tabular}
\end{table}

\section{Conclusion}
In this paper, we  propose three  finite difference schemes for the fractional diffusion-wave equation   \eqref{eq:subeq}.
The first one is based on  the fractional trapezoidal formula in time and the central difference in the space.
This scheme is  proven to be  stable by Fourier analysis with convergence order two in  both time and space.
The second scheme  is based on a  second-order generalized Newton-Gregory formula in time.
The second scheme is only conditionally stable, while
a slight modification of the second scheme   leads to the third scheme that is  stable.
The last two schemes are also  of order two in both time and space.
We extend the two of these time discretization techniques to
a class of fractional differential equations,
and  derived  stable schemes with second-order convergence   both in time and space.

When $\beta\to2$, the present methods \eqref{fem1}, \eqref{fem2-0-2}, and \eqref{fdm2}  becomes the corresponding  classical methods
for the classical  diffusion-wave equation, 
which is an important feature different from the time discretization techniques used in previous papers, see for example
\cite{CuestaLubich2006,DingLi2013,HuangTang2013,JinLazarovZhou2014,McLeanMustapha2007,YangHuang2014}.

We present  numerical experiments to verify the theoretical analysis, and   comparisons with
other methods   exhibit  better accuracy than many of the existing numerical methods.
The present methods  can be readily
extended to   two- and three-dimensional problems and the stability and convergence analysis
are similar to those given here.




\end{document}